\documentclass[journal,twoside,web]{ieeecolor}
\usepackage{generic}
\usepackage{cite}
\usepackage{amsmath,amssymb,amsfonts}
\usepackage{algorithmic}
\usepackage{graphicx}
\usepackage{textcomp}
\usepackage{subcaption}

\newtheorem{assumption}{Assumption}

\newtheorem{theorem}{Theorem}

\def\BibTeX{{\rm B\kern-.05em{\sc i\kern-.025em b}\kern-.08em
    T\kern-.1667em\lower.7ex\hbox{E}\kern-.125emX}}
\begin{document}
\title{Resilient Extremum Seeking Control for Cyber-Physical Systems Under Denial-of-Service Attacks}
\author{Filipe da Silva Bastos Teixeira, Pedro~Henrique~Silva~Coutinho, Tiago Roux Oliveira, and Miroslav Krstic
\thanks{F. S. B. Teixeira is with the Graduate Program in Electronic Engineering, Rio de Janeiro State University, Brazil (e-mail: filipedasilvabastosteixeira51@gmail.com).}
\thanks{P. H. S. Coutinho and T. R. Oliveira are with the Department of Electronics and Telecommunication Engineering, Rio de Janeiro State University, Brazil (e-mails: phcoutinho@eng.uerj.br, tiagoroux@uerj.br).}
\thanks{M. Krstic is with the Department of Mechanical and Aerospace Engineering, University of California San Diego, USA (e-mail: mkrstic@ucsd.edu).} 
\thanks{Corresponding author: Tiago Roux Oliveira.}}

\maketitle

\begin{abstract}
Extremum seeking control (ESC) relies on deliberately injected excitation
to extract optimization information from measured outputs, making its
networked implementation particularly vulnerable to denial-of-service
(DoS) attacks. This paper develops a resilient discrete-time ESC
architecture for cyber-physical systems subject to DoS attacks. By holding the most recently successfully transmitted
signals during DoS intervals, the proposed mechanism yields averaged
error dynamics with an exponentially contracting mode under successful
communication and a neutral mode under attack. For deterministic DoS
attacks, practical exponential convergence to the extremum is established
under an average bound on the attack duration, with an explicit
convergence rate depending on the fraction of time under attack. For
probabilistic DoS attacks, almost-sure and in-probability convergence
properties are established via stochastic averaging, with the attack
success probability explicitly entering the convergence rate. A key
finding is that a seemingly natural zero-input strategy can fundamentally
compromise ESC: attacks synchronized with the dither generate an
$\mathcal{O}(1/a)$ bias in the averaged dynamics, where $a$ is the dither
amplitude, and can displace the equilibrium from the true optimizer.
Thus, under the proposed hold-based architecture, increasingly severe
DoS attacks primarily slow the optimization process rather than destroy
its stability, provided that communication is not permanently blocked.
Numerical simulations illustrate the theoretical guarantees and the
failure mechanism induced by dither-synchronized attacks.
\end{abstract}


\section{Introduction}
\label{sec:introduction}

Cyber-physical systems (CPSs) increasingly rely on communication
networks to close feedback loops between spatially separated sensing,
computation, and actuation units~\cite{sanchez2019bibliographical}. While such architectures enable
flexibility, scalability, and remote operation, they also expose the
feedback loop to disruptions originating from the communication
infrastructure itself. Among cyber threats, denial-of-service (DoS)
attacks are particularly relevant to networked control systems because
an attacker need not manipulate the transmitted information but merely
prevent its delivery. The resulting loss of measurements or control
commands can degrade performance or even destroy closed-loop stability. 
A substantial body of work has therefore investigated resilient
networked control under packet losses, communication constraints, and
malicious attacks, including the analysis of attack impact and
worst-case adversarial strategies in networked estimation and control;
see, e.g., \cite{schenato2007foundations,dePersis2015iss,feng2017resilient,guo2019worstcase,milosevic2020estimating,cetinkaya2016networked,dolk2016event,persis2016networked,coutinho2023switching,coutinho2025resilient}.


The effect of communication disruption becomes fundamentally different
when feedback is used not only for stabilization but also for
\emph{real-time optimization}. Extremum seeking control (ESC) is a
model-free optimization methodology that steers an unknown system toward
an optimizer using measurements of its performance output and a
deliberately injected excitation signal \cite{Oliveira_Krstic_book_2022}. In classical perturbation-based
ESC, the excitation probes the unknown map, while demodulation of the
measured response generates an estimate of its gradient. Rigorous
averaging-based foundations for discrete-time ESC were established in
\cite{choi2002extremum}, with subsequent extensions to multivariable and
stochastic settings, among others, 
\cite{ghaffari2012multivariable,liu2015stochastic}. This ability to
optimize an unknown map directly from measured performance makes ESC
particularly attractive for cyber-physical implementations in which an
accurate model is unavailable or changes during operation.

More broadly, the interplay between control performance and communication
resources has motivated cross-layer designs in which control, sampling,
and networking decisions are treated jointly
\cite{klugel2020joint}. Within extremum seeking, the integration of ESC
with communication-constrained architectures has recently received
increasing attention. 
%
%
In \cite{premaratne2020sporadic}, ESC with sporadic packet transmission was introduced specifically for networked control systems, with the objective
of reducing the communication burden associated with the optimization
process. Event-triggered implementations of perturbation-based ESC have
also been developed to reduce control updates and communication demands
while preserving practical convergence to the extremum
\cite{rodrigues2025event}. These results demonstrate that ESC can be
redesigned to operate efficiently under limited communication resources.
They address, however, communication constraints arising from the
implementation itself rather than \emph{adversarial} communication
losses. In particular, they do not consider a malicious agent capable of
selecting packet losses so as to interact with the periodic excitation
that generates the extremum-seeking gradient estimate.


Within this context, extremum seeking control under gradient deception attacks was investigated using a switching-systems approach~\cite{galarza2022extremum}. While deception attacks inject false data to disrupt CPS operation, DoS attacks prevent information exchange by blocking communication channels. For distributed Nash equilibrium seeking, a resilient strategy was proposed to cope with communication losses among agents caused by DoS attacks~\cite{shao2022distributed}, and~\cite{tang2025deception} addressed deception in Nash equilibrium. Related resilience issues have also been investigated in source-seeking problems. In particular,~\cite {du2024resilient} investigated distributed source seeking for multi-vehicle systems under Sybil attacks, while~\cite{fu2021resilient} studied cooperative source seeking for double-integrator multi-robot systems under deception attacks. More recently, a distributed resilient source-seeking strategy was proposed for mixed cyberattacks, combining deception and DoS attacks on inter-agent communication~\cite{li2025distributed}. Despite these advances, none of the aforementioned works addresses DoS attacks that disrupt the communication channel between the plant and the controller.

This distinction is consequential. When ESC is closed through a
communication network, the same mechanism that enables model-free
optimization creates a vulnerability that is absent from conventional
networked feedback. The information required by perturbation-based ESC
is encoded in the correlation between the response of the unknown map
and the probing signal. Packet losses can therefore do more than remove
isolated measurements or control updates: depending on how unavailable
data are handled, the attack pattern can modify the correlation from
which the gradient estimate is constructed. Consequently, resilience
results developed for stabilization of networked control systems under
DoS attacks cannot, in general, be directly transferred to an
extremum-seeking loop.

In fact, a malicious agent can exploit the periodic probing structure
instead of simply maximizing the number of blocked transmissions.
Consider, for instance, the seemingly natural policy of setting an
unavailable controller-to-plant command to zero. An attack synchronized
with selected portions of the dither cycle then produces a nonzero
cross-correlation between the attack signal and the ESC demodulation
signal. As shown in this paper, the resulting averaged gradient contains
an attack-induced bias whose dominant term is of order
$\mathcal{O}(1/a)$, where $a$ denotes the dither amplitude. Hence,
reducing the excitation amplitude, as normally required for improving
the accuracy of perturbation-based ESC, does not attenuate this effect;
it amplifies it. The true optimizer may consequently cease to be an
equilibrium of the averaged closed-loop system.

Motivated by this observation, we develop a resilient discrete-time ESC
architecture for networked cyber-physical systems subject to DoS
attacks. The key mechanism is simple but structurally important: when
communication is interrupted, the signal applied to the unknown map is
held at its most recently successfully transmitted value. With the
corresponding packet-loss treatment in the feedback channel, this
mechanism prevents the attack signal from generating the harmful
dither--attack correlation described above. More importantly, the
resulting averaged estimation-error dynamics admit a switched
representation in which successful transmissions generate an
exponentially contracting mode, whereas DoS intervals generate a neutral
mode. Hence, resilience can be characterized in terms of the amount
of communication that remains available, without redesigning the
underlying extremum-seeking algorithm.

The main contributions of this paper are summarized as follows:

\begin{itemize}

    \item We introduce a resilient networked ESC architecture under DoS
    attacks and derive its discrete-time averaged closed-loop dynamics.
    The proposed packet-loss handling mechanism transforms communication
    failures into a neutral mode of the averaged dynamics,
    while successful transmissions retain the contracting ESC mode.

    \item For deterministic DoS attacks, we establish practical
    exponential convergence of both the parameter and the performance
    output under an average bound on the attack duration. The result
    explicitly quantifies the degradation of the convergence rate with
    the fraction of time under attack and guarantees resilience for
    every attack fraction strictly smaller than one.

    \item For probabilistic DoS attacks modeled by a Bernoulli process,
    we establish stochastic convergence by means of discrete-time
    stochastic averaging. In particular, the extremum-seeking error is
    characterized almost surely and in probability, with an explicit
    dependence of the averaged convergence rate on the attack success
    probability.
    
    \item We uncover a vulnerability specific to perturbation-based ESC
    that is hidden in conventional packet-loss analyses. A DoS attack
    correlated with the dither creates an $\mathcal{O}(1/a)$ bias in the
    averaged dynamics under a zero-input policy and can displace the
    closed-loop equilibrium from the true optimizer. This result explains
    why packet-loss compensation cannot be designed independently of the
    probing mechanism and why the proposed hold-input strategy remains
    resilient even against dither-synchronized attacks.

\end{itemize}

An important consequence of these results is that, under the proposed
architecture, the severity of a DoS attack manifests itself primarily
through the \emph{speed of optimization} rather than through loss of
stability. Increasing the deterministic attack fraction or the
attack success probability slows convergence, but does not destroy
it as long as successful communication is not eliminated altogether.
Numerical examples corroborate the theoretical bounds for both
deterministic and probabilistic DoS models and illustrate how
dither-synchronized attacks can compromise a conventional zero-input
implementation while the proposed resilient strategy preserves
convergence to the true extremum.

The remainder of the paper is organized as follows.
Section~\ref{sec:problem} formulates the networked extremum-seeking
problem under DoS attacks. Sections~\ref{sec:main-results-1} and
\ref{sec:main-results-2} present the deterministic and probabilistic
resilience results, respectively. The disruptive effect of
dither-synchronized DoS attacks is analyzed in
Section~\ref{sec:correlation_analysis}. Numerical results are presented
in Section~\ref{sec:results}, and Section~\ref{sec:conclusion} concludes
the paper.

\textit{Notation:} The sets of real numbers, nonnegative integers, and
positive integers are denoted by $\mathbb{R}$, $\mathbb{N}_0$, and
$\mathbb{Z}^{+}$, respectively. For a scalar $x\in\mathbb{R}$, $|x|$
denotes its absolute value, whereas $|\mathcal{S}|$ denotes the cardinality
of a set $\mathcal{S}$. The expectation and probability operators are
denoted by $\mathbb{E}\{\cdot\}$ and $\mathbb{P}\{\cdot\}$, respectively.
For $x\in\mathbb{R}$, $\lfloor x\rfloor$ denotes the greatest integer not
exceeding $x$. The notation $\mathcal{O}(\varepsilon)$ denotes a quantity
whose magnitude is bounded by $C\varepsilon$ for some constant $C>0$
independent of $\varepsilon$, as $\varepsilon\to0$.

%
%
%
%
%
%
%

\section{Problem Formulation}
\label{sec:problem}

Consider the networked ESC system shown in Fig.~\ref{fig:block_diagram}. 
\begin{figure}[!ht]
    \centering
    \includegraphics[width=\linewidth]{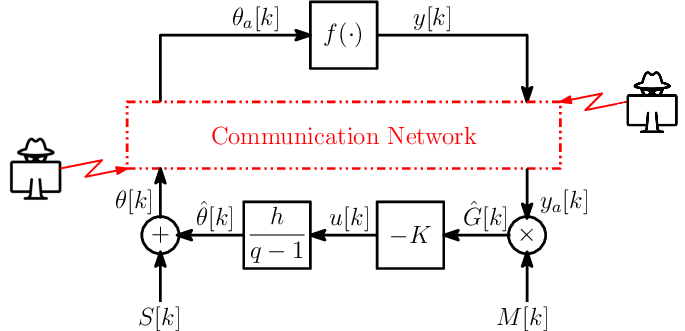}
    \caption{Networked extremum-seeking control system under DoS attacks.}
    \label{fig:block_diagram}
\end{figure}

In this system, the aim is to drive the output $y[k]$ to an optimal point $f^\ast$ of the nonlinear static map $f(\cdot)$ by appropriately adjusting its input. 
Without loss of generality, the static nonlinear map $f(\theta)$ is assumed to achieve an optimum at $\theta = \theta^\ast$, and to be of the following form
\begin{align}\label{eq:f}
    f(\theta) = f^\ast + \frac{H}{2}(\theta - \theta^\ast)^2,
\end{align}
where $H = [\partial^2 f / \partial \theta^2]_{\theta = \theta^\ast}$ is the Hessian, which is positive if the optimum is a minimum and negative if the optimum is a maximum. In this work, we assume that the sign of the Hessian $H$ is known. Note that cubic and higher-order terms in~\eqref{eq:f} are omitted, as they are negligible in local stability analysis via averaging~\cite{choi2002extremum}.
Moreover, the dither signals are defined by~\cite{ghaffari2012multivariable}:
\begin{align}\label{eq:dithers}
    S[k] = a \sin{(\omega h k)}, \quad M[k] = \frac{2}{a} \sin{(\omega h k)},
\end{align}
where $a > 0$ is the amplitude, $h > 0$ is the sampling period, and $\omega > 0$ is the frequency.

In contrast to a standard ESC system, in which the output and the input of the map $f(\cdot)$ are continuously transmitted between the map and adaptation law of the ESC system, in the considered setup, the output of the map $y[k]$ is transmitted to the adaptation law via a general purpose communication network, such that the output signal available is $y_a[k]$. Moreover, the input signal computed by the adaptation law, $\theta[k]$, is also transmitted to the map via the communication network, so that the signal effectively applied to the map is $\theta_a[k]$. 
In this work, we assume the communication network is subject to DoS attacks by a malicious agent. Thus, the signals transmitted through the network, $y[k]$ and $\theta[k]$, can be blocked. \textcolor{black}{To handle the effects of DoS attacks on communication, there are two main strategies: zero-input and hold-input~\cite{schenato2007foundations}. In the former, the transmitted signal is set to zero if the corresponding data packet is not received, whereas in the latter, the most recently received value is retained until a new data packet becomes available. In this work, we adopt the hold-input strategy for both communication channels. Specifically, whenever a transmission is unsuccessful, the most recently successfully received value is retained until a new packet becomes available, such that}
%
%
\begin{align}
    y_a[k] &= 
    \begin{cases}
        y[k], & \text{if transmission is successful},\\
        y_a[k-1], & \text{otherwise},
    \end{cases} \\
     \theta_a[k] &= 
    \begin{cases}
        \theta[k], & \text{if transmission is successful}, \\
        \theta_a[k-1], & \text{otherwise}.
    \end{cases}
\end{align}

Let $\sigma[k] \in \{0,1\}$ be a binary-valued variable that indicates success/failure of packet exchange attempts, such that
\begin{align}\label{eq:sigma}
    \sigma[k] = \begin{cases}
        1, & \text{successful transmission at sample $k$}, \\
        0, & \text{otherwise}.
    \end{cases}
\end{align}
As discussed in~\cite{schenato2007foundations}, TCP-like and UDP-like communication protocols equipped with an acknowledgment mechanism can provide information about $\sigma[k]$ to determine the instants at which DoS attacks block transmission attempts. Based on $\sigma[k]$, we can write 
\begin{align}
    y_a[k] &= \sigma[k] y[k] + (1-\sigma[k]) y_a[k-1], \label{eq:y_a} \\
    \theta_a[k] &= \sigma[k] \theta[k] + (1-\sigma[k]) \theta_a[k-1]. \label{eq:theta_a}
\end{align}

Let $\tilde{\theta}[k]$ be the estimation error defined by
\begin{align}\label{eq:tilde_theta}
    \tilde{\theta}[k] = \hat{\theta}[k] - \theta^*,
\end{align}
which is equivalent to
\begin{align}
    \tilde{\theta}[k] = \hat{\theta}[k] - \theta^\ast = \frac{h}{q-1}u[k] - \theta^\ast.
\end{align}
By multiplying both sides by $q-1$ and noticing that $(q-1)\theta^* = 0$, since ${\theta}^*[k+1] = {\theta}^\ast[k] = \theta^\ast$, the error dynamics becomes
\begin{align}\label{eq:error_dynamics_1}
    \tilde{\theta}[k+1] = \tilde{\theta}[k] + h u[k],
\end{align}
where
\begin{align}\label{eq:u}
    u[k] = - K \hat{G}[k],
\end{align}
and
\begin{align}\label{eq:gradient}
    \hat{G}[k] = M[k] y_a[k]
\end{align}
is the gradient estimate, which depends on the output signal under DoS attack. Using~\eqref{eq:y_a}, \eqref{eq:u} can be still expressed as
\begin{align}\label{eq:u_a}
    u[k] = - \sigma[k] K M[k] y[k].
\end{align}
From \eqref{eq:tilde_theta} and the relation
\begin{align}
    \theta[k] = \hat{\theta}[k] + S[k],
\end{align}
we have that the input provided by the ESC law is
\begin{align}\label{eq:theta_relation}
    \theta[k] = \tilde{\theta}[k] + S[k] + \theta^\ast.
\end{align}
Moreover, the output of the map under the input subject to DoS attacks $\theta_a[k]$ is
\begin{align}\label{eq:f_attack}
    y[k] &= f(\theta_a[k]) = f^\ast + \frac{H}{2} (\theta_a[k] - \theta^\ast)^2.
\end{align}

Using~\eqref{eq:y_a}, \eqref{eq:theta_a}, \eqref{eq:u_a}, \eqref{eq:theta_relation}, and~\eqref{eq:f_attack}, the closed-loop error dynamics \eqref{eq:error_dynamics_1} becomes
\begin{align}
    \tilde{\theta}[k+1] &= \tilde{\theta}[k] - h K \sigma[k] M[k] f^* \nonumber \\
    &- \frac{h K H}{2} \sigma^2[k] M[k] \left( \tilde{\theta}^2[k] + 2\tilde{\theta}[k] S[k] + S^2[k] \right) \nonumber \\
    &- \frac{h K H}{2} M[k] \sigma[k](1 - \sigma[k])(\theta_a[k-1] - \theta^*)^2 \nonumber \\
    & - h K M[k](1-\sigma[k])y_a[k-1].\label{eq:error_dynamics_2} 
\end{align}
Provided that $\sigma^2[k] = \sigma[k]$ and $\sigma[k] (1-\sigma[k]) = 0$, for any $k \in \mathbb{N}_0$, we obtain
\begin{align}
    \tilde{\theta}[k+1] &= \tilde{\theta}[k] + h F_\sigma(k,\tilde{\theta}[k],h), \label{eq:error_dynamics_3}
\end{align}
where
\begin{align}
    F_\sigma(k,\tilde{\theta}[k],h) &= - K \sigma[k] M[k] f^* - \frac{K H}{2} \sigma[k] M[k]  \tilde{\theta}^2[k] \nonumber \\
    &- \frac{K H}{2} \sigma[k] M[k] \left( 2\tilde{\theta}[k] S[k] + S^2[k]\right) \nonumber \\
    & - h K M[k](1-\sigma[k])y_a[k-1]. \label{eq:F_sigma}
\end{align}
It is evident that the closed-loop error dynamics~\eqref{eq:error_dynamics_3} constitute a switched system with two modes, governed by $\sigma[k] \in \{0,1\}$.
The aim of this work is to develop conditions for analyzing the stability of the closed-loop error dynamics~\eqref{eq:error_dynamics_3} in the presence of deterministic and probabilistic DoS attacks. The following sections discuss the specific formulations for each case.

\section{Resilient Extremum Seeking Control Against Deterministic DoS Attacks}
\label{sec:main-results-1}

\subsection{Deterministic DoS Attack Model}

This section presents the deterministic DoS attack model adopted in this work. Let  $\{H_n\}_{n \in \mathbb{Z}^+}$ be the sequence of attacks, where the $n$-th attack is defined by:
\begin{equation}
    H_n \triangleq \{h_n, \dots, h_n + \tau_n - 1\} \label{eq:H_n}
\end{equation}
where $h_n \in \mathbb{Z}^+$ is the attack start time and  $\tau_n \in \mathbb{Z}^+$ is the attack duration. The set of all time instants under attack is:
\begin{equation}
\mathcal{T} \triangleq \bigcup_{n \in \mathbb{Z}^+} H_n.
\end{equation}
It is assumed that the total attack duration satisfies the following assumption, as in \cite{van2020towards,coutinho2025resilient}.
\begin{assumption}\label{assump:duration}
Let $|\Xi(k_1,k_2)|$ be the total attack duration over an interval $\{k_1, k_1+1,\ldots, k_2-1\}$, with $\Xi(k_1,k_2) = \mathcal{T} \cap \{k_1, \dots, k_2-1\}$. The attack duration satisfies
\begin{equation}\label{eq:DoS_duration}
    |\Xi(k_1,k_2)| \leq \eta + \nu (k_2 - k_1), \quad \forall k_1 < k_2, \, k_1,k_2 \in \mathbb{Z}^+,
\end{equation}
where $\eta \geq 0$ is an offset parameter related to the attack duration at $k = 0$ and $\nu \in [0,1]$ is the fraction of time under attack over the interval.
\end{assumption}

\subsection{Average Closed-Loop System}
\label{sec:2-B}

The closed-loop system \eqref{eq:error_dynamics_3} depends on a small parameter $h$, and the discrete-time vector field $F_\sigma(k,\tilde{\theta},h)$ is $T$-periodic in $k$, with 
\begin{align}
T&= \frac{2\pi}{\omega h}\,. \label{eq:T}
\end{align}  
Consequently, the discrete-time averaging method \cite{bai1988averaging,choi2002extremum,plotnikov2004averaging} can be applied to each mode of $F_\sigma(k,\tilde{\theta},h)$ as $h \to 0$.
The averaging approach enables characterization of the resulting average autonomous system, which approximates the original non-autonomous system \eqref{eq:error_dynamics_3}. Intuitively, when the system evolves on a time scale slower than that of the excitation, its behavior is mainly determined by the excitation's average effect over one period.

For each fixed communication mode \(\sigma\in\{0,1\}\), the corresponding vector field is periodic with respect to the dither signals. We therefore average the two mode-dependent vector fields separately and subsequently reconstruct the switched averaged system using the original switching signal \(\sigma[k]\).
Applying the discrete-time averaging technique~\cite{bai1988averaging,choi2002extremum,plotnikov2004averaging} to each mode of \eqref{eq:error_dynamics_3} yields the average system
\begin{align}
\tilde{\theta}_{\rm{av}}[k+1] &= \tilde{\theta}_{\rm{av}}[k] + h F_\sigma^{\rm{av}}(k,\tilde{\theta}_{\rm{av}},h), \label{eq:xav_k+1_event} \\
F_\sigma^{\rm{av}}(k,\tilde{\theta}_{\rm{av}},h) &= \lim_{h \to 0}\lim_{T \to \infty} \frac{1}{T} \sum_{k=s+1}^{s+T} F_\sigma(k,\tilde{\theta}_{\rm{av}},h). \label{eq:fav_event}
\end{align}
Therefore, by ``freezing'' the average variable $\tilde{\theta}_{\rm{av}}$ in \eqref{eq:error_dynamics_3}--\eqref{eq:F_sigma}, we obtain the following autonomous switched average system
\begin{align}
\tilde{\theta}_{\rm{av}}[k+1]&=\left(1 - \sigma[k] h H K\right)\tilde{\theta}_{\rm{av}}[k]. \label{eq:error_dynamics_av}
\end{align}
Based on~\eqref{eq:error_dynamics_av}, the following assumption is made.
\begin{assumption}\label{assump:av_stability}
There exist $h > 0$ and $K \in \mathbb{R}$ such that 
    \begin{align}\label{eq:av_stability}
        |1 - h H K| < 1.
    \end{align}
\end{assumption}

The switched average closed-loop dynamics in~\eqref{eq:error_dynamics_av} reveal that, during DoS attacks ($\sigma[k]= 0$), the system enters a neutral mode where the error does not change, i.e., $\tilde{\theta}[k+1] = \tilde{\theta}[k]$. Conversely, in the absence of attacks ($\sigma[k] = 1$), the condition established in Assumption~\ref{assump:av_stability} ensures that the system~\eqref{eq:error_dynamics_av} evolves in a stable mode, as ensured by~\eqref{eq:av_stability}.

\subsection{Stability Analysis}

The following theorem states the stability analysis of the closed-loop system.

\begin{theorem} \label{thm:stab_dos}
Consider the closed-loop system \eqref{eq:error_dynamics_3}--\eqref{eq:F_sigma}
with the dither signals $S[k]$ and $M[k]$ defined in~\eqref{eq:dithers} and $\sigma[k] \in \{0,1\}$ defined according to~\eqref{eq:sigma}. Suppose that the sequence of DoS attacks satisfies Assumption~\ref{assump:duration} with
\begin{align}\label{eq:bar_nu}
    \nu \leq \bar{\nu}, \quad \bar{\nu} \in [0,1)
\end{align}
and that Assumption~\ref{assump:av_stability} holds. Then, for $h > 0$ sufficiently small, the equilibrium $\tilde{\theta}_{\rm av} = 0$ of the average system~\eqref{eq:error_dynamics_av} is exponentially stable, and the following inequalities hold for the non-average system:
\begin{align}
|\theta[k] - \theta^*| &\leq \rho e^{\lambda k/2} |\theta[0] - \theta^*| + \mathcal{O}(a + h), \label{eq:theta_convergence}
\end{align}
\begin{align}
|y[k] - f^*| &\leq 2 \rho^2 e^{\lambda k} |y[0] - f^*| + \mathcal{O}(a^2 + h^2),\label{eq:y_convergence}
\end{align}
where $\rho = \sqrt{\beta^{-\eta}}$, $\beta = (1-hHK)^2$ and $\lambda = \ln(\beta)(1 - \nu)$.
\end{theorem}

\begin{proof}
The proof has two parts. First, we show the exponential stability of the average system, and then we show the practical exponential stability of the original system via the averaging theorem.

\subsubsection*{Exponential Stability of the Average System}

Consider the Lyapunov function candidate for the average system \eqref{eq:error_dynamics_av}:
\begin{align}
V(\tilde{\theta}_{\rm av}[k]) = \tilde{\theta}_{\rm av}^2[k]. \label{eq:lyap}
\end{align}
We evaluate the Lyapunov function candidate along the trajectories of each mode of the average system \eqref{eq:error_dynamics_av}. 

Let $\Xi(0, k)$ denote the set of attack instants up to time $k$, and let $\Theta(0, k)$ denote its complement, i.e., the set of samples of successful communication up to time $k$:
\begin{align}
\Theta(0, k) = \{0, 1, \dots, k-1\} \setminus \Xi(0, k).
\end{align}
By definition, the cardinalities satisfy:
\begin{align}
|\Theta(0, k)| = k - |\Xi(0, k)|. \label{eq:normal_bound}
\end{align}
Moreover, since for any $k_1, k_2 \geq 0$, with $k_2 \geq k_1$, we have $\Theta(k_1,k_2) \cup \Xi(k_1,k_2)$, we can write
\begin{align}
    \Theta(k_1,k_2) = \bigcup_{i \in \mathbb{N}} \mathcal{Z}_i \cap \{k_1,\ldots,k_2\} \label{eq:Theta} \\
    \Xi(k_1,k_2) = \bigcup_{i \in \mathbb{N}} \mathcal{W}_{i-1} \cap \{k_1,\ldots,k_2\} \label{eq:Xi}
\end{align}
where 
\begin{align}
    \mathcal{Z}_i &= \{\ell_i,\ldots, \ell_i + v_i - 1\}, \label{eq:Zi} \\
    \mathcal{W}_{i} &= \{\ell_i+v_i,\ldots, \ell_{i+1}-1\}. \label{eq:Wi}
\end{align}
Here, $\ell_i$ is the starting instant of the $i$-th successful transmission interval, and $v_i \ge 1$ is its duration. Consequently, $\mathcal{W}_i$ represents the subsequent interval of unsuccessful transmissions. If $h_0 > 0$, then $\ell_0 = h_0$ and $\mathcal{W}_{-1} = \{0,\ldots, \ell_0-1\}$; if $h_0 = 0$, then $\mathcal{W}_{-1} = \emptyset$. Based on these definitions, $k \in \mathcal{Z}_i \Rightarrow \sigma[k] = 1$ and $k \in \mathcal{W}_i \Rightarrow \sigma[k] = 0$.

Then, if the communication is successful and $k \in \mathcal{Z}_i$, we have $\sigma[k] = 1$ and the system operates in the mode governed by:
\begin{align}
\tilde{\theta}_{\rm av}[k+1] = (1 - h H K) \tilde{\theta}_{\rm av}[k], \quad \forall k \in \mathcal{Z}_i.
\end{align}
Therefore, for any interval $\mathcal{Z}_i$, we have
\begin{align}
    V(\tilde{\theta}_{\rm av}[\ell_i+1]) = (1 - h H K)^2 V(\tilde{\theta}_{\rm av}[\ell_i]) \triangleq \beta V(\tilde{\theta}_{\rm av}[\ell_i]). \label{eq:beta_def}
\end{align}
Hence, by evaluating~\eqref{eq:beta_def} recursively for $k \in \mathcal{Z}_i$, we obtain
\begin{align}
    V(\tilde{\theta}_{\rm av}[k]) = \beta^{k - \ell_i} V(\tilde{\theta}_{\rm av}[\ell_i]). \label{eq:beta_def_2}
\end{align}
Under Assumption~\ref{assump:av_stability}, it follows that $0 < \beta < 1$ and, consequently, $V(\tilde{\theta}_{\rm av}[k]) \leq \beta^{k - \ell_i} V(\tilde{\theta}_{\rm av}[\ell_i])$.

If the communication is unsuccessful due to the presence of DoS attacks and $k \in \mathcal{W}_i$, we have $\sigma[k] = 0$, and the system operates in the mode governed by:
\begin{align}
\tilde{\theta}_{\rm av}[k+1] = \tilde{\theta}_{\rm av}[k],
\end{align}
Thus, for any interval $\mathcal{W}_i$, we have
\begin{align}
V(\tilde{\theta}_{\rm av}[\ell_i+v_i+1]) = V(\tilde{\theta}_{\rm av}[\ell_i+v_i]). \label{eq:neutral_mode}
\end{align}
By evaluating~\eqref{eq:beta_def} recursively for $k \in \mathcal{W}_i$, we obtain
\begin{align}
    V(\tilde{\theta}_{\rm av}[k]) = V(\tilde{\theta}_{\rm av}[\ell_i+v_i]). \label{eq:neutral_mode_2}
\end{align}

Applying \eqref{eq:beta_def_2} and \eqref{eq:neutral_mode_2} recursively from $j\in\{0,\ldots,k-1\}$, using \eqref{eq:normal_bound}, and noting that the state $V(\tilde{\theta}_{\rm av}[k])$ depends on $k$ transitions, we obtain
\begin{align}
    V(\tilde{\theta}_{\rm av}[k]) = \beta^{k - |\Xi(0, k)|} 1^{|\Xi(0, k)|} V(\tilde{\theta}_{\rm av}[0]). \label{eq:V_recursion}
\end{align}
Using \eqref{eq:DoS_duration} to bound $|\Xi(0, k)| \leq \eta + \nu k$ in \eqref{eq:V_recursion}, yields
\begin{align}
V(\tilde{\theta}_{\rm av}[k]) \leq 
\beta^{-\eta} e^{\lambda k} V(\tilde{\theta}_{\rm av}[0]), \quad \forall k \in \mathbb{N}_0, \label{eq:V_bound}
\end{align}
where $\lambda = \ln(\beta)(1 - \nu)$. Therefore, from \eqref{eq:V_bound}:
\begin{align}
|\tilde{\theta}_{\rm av}[k]| \leq \sqrt{\beta^{-\eta}} e^{\lambda k/2} |\tilde{\theta}_{\rm av}[0]|. \label{eq:av_convergence}
\end{align}
Since $0 < \beta < 1$ and $\ln(\beta) < 0$, condition~\eqref{eq:bar_nu} ensures that $\lambda < 0$. Thus, it follows that the origin of the average system \eqref{eq:error_dynamics_av} is exponentially stable.

\subsubsection*{Practical Exponential Stability of the Original System}


Since the averaging approximation is applied mode-wise and the resulting averaged switched system is uniformly exponentially stable over all DoS sequences satisfying Assumption~\ref{assump:duration} and (\ref{eq:bar_nu}), the averaging argument can be applied consistently over this admissible class of switching signals. Moreover, since the vector field $F_\sigma(k, \tilde{\theta},h)$ in~\eqref{eq:F_sigma} is $T$-periodic in $k$ with $T = 2\pi/(\omega h)$ and the average system with state variable~$\tilde{\theta}_{\rm av}[k]$ is exponentially stable, it follows from the averaging theorem in~\cite{bai1988averaging,choi2002extremum,plotnikov2004averaging} that, for $h > 0$ sufficiently small and initial condition $\tilde{\theta}[0]$ sufficiently close to the origin, the solutions of the original system~\eqref{eq:error_dynamics_3} locally exponentially converge to an $\mathcal{O}(h)$ neighborhood such that
\begin{align}
|\tilde{\theta}[k] - \tilde{\theta}_{\rm av}[k]| \leq \mathcal{O}(h). \label{eq:avg_error}
\end{align}

From \eqref{eq:av_convergence} and \eqref{eq:avg_error}, we have that
\begin{align}
|\tilde{\theta}[k]| &\leq |\tilde{\theta}_{\rm av}[k]| + |\tilde{\theta}[k] - \tilde{\theta}_{\rm av}[k]| \nonumber \\
&\leq \sqrt{\beta^{-\eta}} e^{\lambda k/2} |\tilde{\theta}_{\rm av}[0]| + \mathcal{O}(h). \label{eq:theta_bound}
\end{align}
Now, using $\theta[k] = \tilde{\theta}[k] + S[k] + \theta^*$, with $S[k] = a \sin(\omega k h)$, we have
\begin{align}
|\theta[k] - \theta^*| \leq |\tilde{\theta}[k]| + |S[k]|.
\end{align}
Then, from \eqref{eq:theta_bound} and the fact that $S[k]$ is of order $\mathcal{O}(a)$, 
we obtain~\eqref{eq:theta_convergence}.

For the output $y[k] = f(\theta_a[k])$, using the quadratic map expression under DoS attacks in~\eqref{eq:f_attack}, we have
\begin{align}
|y[k] - f^*| &= \frac{|H|}{2} |\theta_a[k] - \theta^*|^2 \nonumber \\
            &= \frac{|H|}{2} |\sigma[k]\theta[k] + (1-\sigma[k])\theta_a[k-1] - \theta^*|^2. \label{eq:y_eq}
\end{align}
Due to the dependence on the operation mode associated with $\sigma[k]$, we proceed by evaluating~\eqref{eq:y_eq} for both cases. If the communication is successful and $k \in \mathcal{Z}_i$, we have $\sigma[k] = 1$ and
\begin{align}
    |y[k] - f^*| &= \frac{|H|}{2} |\theta[k] -\theta^*|^2 \nonumber \\
    &\leq \frac{|H|}{2} \beta^{-\eta} e^{\lambda k}|\theta[0] - \theta^*|^2 + [\mathcal{O}(a+h)]^2 \nonumber \\
    & + |H| \sqrt{\beta^{-\eta}} e^{\lambda k/2}|\theta[0] - \theta^*|\mathcal{O}(a+h), \; \forall k \in \mathcal{Z}_i.\label{eq:y_ineq_1}
\end{align}
By applying the Young's inequality $c d \leq \varepsilon c^2/2 + d^2/(2\varepsilon)$, for $c,d,\varepsilon > 0$, with $c = \sqrt{\beta^{-\eta}} e^{\lambda k/2}|\theta[0] - \theta^*|$, $d = \mathcal{O}(a+h)$ and $\varepsilon = 1$, the inequality~\eqref{eq:y_ineq_1} is upper bounded by
\begin{align}
    |y[k] - f^*| &\leq 2 \beta^{-\eta} e^{\lambda k} \left(\frac{|H|}{2}|\theta[0] - \theta^*|^2\right) \nonumber \\
    &+ (1+|H|/2)[\mathcal{O}(a+h)]^2, \; \forall k \in \mathcal{Z}_i. \label{eq:y_ineq_2}
\end{align}
As $\frac{H}{2}|\theta[0] - \theta^*|^2 = |y[0] - f^*|$, from~\eqref{eq:f_attack}, and $(1+|H|/2)[\mathcal{O}(a+h)]^2$ remains in the order of magnitude $\mathcal{O}(a^2 + h^2)$, it follows from~\eqref{eq:y_ineq_2} that~\eqref{eq:y_convergence} holds for all $k \in \mathcal{Z}_i$.

Now, we evaluate the case in which the communication is unsuccessful and $k \in \mathcal{W}_i$, \textit{i.e.}, $\sigma[k] = 0$. For any interval $\mathcal{W}_i$, we have
\begin{align}
    |y[k] - f^*| &= \frac{|H|}{2} |\theta_a[\ell_i+v_i] -\theta^*|^2 \nonumber \\
                 &= \frac{|H|}{2} |\theta_a[\ell_i+v_i-1] -\theta^*|^2 \nonumber \\
                 &= \frac{|H|}{2} |\theta[\ell_i+v_i-1] -\theta^*|^2, \quad \forall k \in \mathcal{W}_i. \label{eq:y_eq_2}
\end{align}
since $\theta_a[\ell_i+v_i] = \theta_a[\ell_i+v_i-1]$, from~\eqref{eq:theta_a}, which corresponds to the last successfully transmitted input $\theta[\ell_i+v_i-1]$ to the map. 

Let $\kappa_i = \ell_i + v_i - 1$ denote the time of the last successful transmission. From~\eqref{eq:y_eq_2}, we have that $|y[k] - f^*| = |y[\kappa_i] - f^*|$, since the applied input remains constant during $\mathcal{W}_i$. 
To show that this constant output satisfies the exponential envelope for $k > \kappa_i$, we analyze the underlying average system. Because no successful transmissions occur during $\mathcal{W}_i$, the number of successful updates up to time $k$ is identical to that up to time $\kappa_i$. Thus, according to~\eqref{eq:neutral_mode_2} and~\eqref{eq:V_recursion}, the Lyapunov function remains constant during the attack, yielding $V(\tilde{\theta}_{\rm av}[\kappa_i]) = V(\tilde{\theta}_{\rm av}[k])$. 

Applying the global duration bound \eqref{eq:DoS_duration} evaluated at the current time $k$ directly to $V(\tilde{\theta}_{\rm av}[k])$, we obtain
\begin{align}
    |\tilde{\theta}_{\rm av}[\kappa_i]| &= |\tilde{\theta}_{\rm av}[k]| \nonumber \\
    &\leq \sqrt{\beta^{-\eta}} e^{\lambda k/2} |\tilde{\theta}_{\rm av}[0]|, \quad \forall k \in \mathcal{W}_i. \label{eq:theta_av_W}
\end{align}
Notice that \eqref{eq:theta_av_W} strictly bounds the past state at $\kappa_i$ using the tighter exponential envelope at time $k$.

Following the practical convergence steps for $\theta[\kappa_i] = \tilde{\theta}[\kappa_i] + S[\kappa_i] + \theta^*$, and using $|S[\kappa_i]| = \mathcal{O}(a)$ with~\eqref{eq:avg_error}, we have
\begin{align}
    |\theta[\kappa_i] - \theta^*| &\leq |\tilde{\theta}_{\rm av}[\kappa_i]| + |\tilde{\theta}[\kappa_i] - \tilde{\theta}_{\rm av}[\kappa_i]| + |S[\kappa_i]| \nonumber \\
    &\leq \sqrt{\beta^{-\eta}} e^{\lambda k/2} |\tilde{\theta}_{\rm av}[0]| + \mathcal{O}(a+h). \label{eq:theta_kc_final}
\end{align}
Substituting~\eqref{eq:theta_kc_final} back into the output error~\eqref{eq:y_eq_2}, we get
\begin{align}
    |y[k] - f^*| &= \frac{|H|}{2} |\theta[\kappa_i] -\theta^*|^2 \nonumber \\
    &\leq \frac{|H|}{2} \left( \sqrt{\beta^{-\eta}} e^{\lambda k/2}|\tilde{\theta}_{\rm av}[0]| + \mathcal{O}(a+h) \right)^2.
\end{align}
By applying Young's inequality ($cd \leq c^2/2 + d^2/2$) identically to the procedure in~\eqref{eq:y_ineq_1}--\eqref{eq:y_ineq_2}, we obtain
\begin{align}
    |y[k] - f^*| &\leq 2 \beta^{-\eta} e^{\lambda k} \left(\frac{|H|}{2}|\theta[0] - \theta^*|^2\right) \nonumber \\
    &+ (1+|H|/2)[\mathcal{O}(a+h)]^2, \quad \forall k \in \mathcal{W}_i. \label{eq:y_ineq_W}
\end{align}
Therefore, the output bound~\eqref{eq:y_convergence} also holds for all $k \in \mathcal{W}_i$. 
Thus,~\eqref{eq:y_convergence} holds for all $k \in \mathbb{N}_0$.
This concludes the proof. 
\end{proof}

\section{Resilient Extremum Seeking Control Against Probabilistic DoS Attacks}
\label{sec:main-results-2}

\subsection{Probabilistic DoS Attack Model}

This section presents the probabilistic DoS attack model employed in this work. Consider the binary variable $\xi[k] \in \{0,1\}$ that indicates whether a packet exchange failed over the communication channels. When $\xi[k] = 0$, the packet exchange attempt at $k$ is successful. On the other hand, $\xi[k] = 1$ means that either the packet sent from the map or the packet sent from the controller is lost at time $k$. Clearly, $\sigma[k]$ in~\eqref{eq:sigma} can be expressed in terms of $\xi[k]$ as $\sigma[k] = 1 - \xi[k]$. 
Based on~\cite{cetinkaya2016networked}, we make the following assumption for the adopted probabilistic DoS attack model.
\begin{assumption}\label{assump:probability}
    Let $p \in [0,1]$ be the probability of successful DoS attack. $\{\xi[k] \in \{0,1\}\}_{k \in \mathbb{N}_0}$ is a Bernoulli process. This means that it is an ergodic process and $\mathbb{E}\{\xi\} = p$.
\end{assumption}

\subsection{Average Closed-Loop System}

Consider the closed-loop system~\eqref{eq:error_dynamics_3}.
By following similar steps as Section~\ref{sec:problem} and~\cite[Section~IV]{liu2015stochastic}, the average closed-loop system can be obtained as 
\begin{align}
\tilde{\theta}_{\rm{av}}[k+1] &= \tilde{\theta}_{\rm{av}}[k] + h F_p^{\rm{av}}(\tilde{\theta}_{\rm{av}}), \label{eq:xav_prob} \\
F_p^{\rm{av}}(\tilde{\theta}_{\rm{av}}) &= \int_{\Xi} F(\tilde{\theta}, \xi) \mu(d\xi) \\
&= \lim_{N \to \infty} \frac{1}{N+1} F(\tilde{\theta}, \xi[k+1])~\mathrm{a.s.},\label{eq:fav_prob}
\end{align}
where $\Xi$ is the living space of the stochastic variable $\xi[k]$.
Provided that $\mathrm{E}[\xi] = p$, we obtain the following average system
\begin{align}
\tilde{\theta}_{\rm{av}}[k+1]&=\left(1 - (1-p) h H K\right)\tilde{\theta}_{\rm{av}}[k]. \label{eq:error_dynamics_av_prob}
\end{align}
Based on~\eqref{eq:error_dynamics_av_prob}, the following assumption is made.
\begin{assumption}\label{assump:av_stability_prob}
There exist $h > 0$ and $K \in \mathbb{R}$ such that
    \begin{align}\label{eq:av_stability_prob}
        |1 - (1-p)h H K| < 1.
    \end{align}
\end{assumption}

\subsection{Stability Analysis}

\begin{theorem} \label{thm:stab_dos_prob}
Consider the closed-loop system \eqref{eq:error_dynamics_3}--\eqref{eq:F_sigma}
with the dither signals $S[k]$ and $M[k]$ defined in~\eqref{eq:dithers} and $\sigma[k] \in \{0,1\}$ defined according to~\eqref{eq:sigma}. Suppose that the sequence of DoS attacks satisfies Assumption~\ref{assump:probability} with $p \in [0,1)$ and that Assumption~\ref{assump:av_stability_prob} holds. Then, for $h > 0$ sufficiently small, the equilibrium $\tilde{\theta}_{\rm av} = 0$ of the average system~\eqref{eq:error_dynamics_av_prob} is exponentially stable almost surely and in probability, in the sense that:
\begin{align}
&\lim_{h \to 0} \inf \{ k \in \mathbb{N} : |\tilde{\theta}[k]| > \alpha^{k/2} |\tilde{\theta}[0]| + \delta\} = +\infty,~\text{a.s.} \label{eq:theta_convergence_prob_k} \\
&\lim_{h \to 0} \mathbb{P} \{ |\tilde{\theta}[k]| \leq \alpha^{k/2} |\tilde{\theta}[0]| + \delta, \forall k = 0,\ldots,\lfloor N/h \rfloor \} = 1, \label{eq:theta_convergence_prob}
\end{align}
where $\alpha = (1-(1-p)hHK)^2$. Moreover, for some $C_y > 0$:
\begin{align}
&|y[k] - f^*| \leq C_y \alpha^{k} |y[0] - f^*| + \mathcal{O}(a^2) + \mathcal{O}(\delta^2), \forall k < \tau_h^\delta,\label{eq:y_convergence_prob_k} \\
&\lim_{h \to 0} \mathbb{P} \{ |y[k] - f^*| \leq C_y\alpha^{k}|y[0] - f^*| + \mathcal{O}(a^2) + \mathcal{O}(\delta^2), \nonumber \\
& \qquad \qquad \qquad \qquad \forall k = 0,1,\ldots, \lfloor N/h \rfloor\} = 1,\label{eq:y_convergence_prob}
\end{align}
where
\begin{align}
    \tau_h^\delta = \inf\{k \in \mathbb{N} : |\tilde{\theta}[k]| > |\tilde{\theta}[0]|\alpha^{k/2} + \delta\}. \label{eq:tau_h}
\end{align}
\end{theorem}

\begin{proof}
The proof is divided into two parts. First, we show the exponential stability of the average system, and then we show the stability of the original system via the stochastic averaging theorem for discrete-time systems~\cite{liu2015stochastic}.

\subsubsection*{Exponential Stability of the Average System}

To analyze the stability of the origin $\tilde{\theta}_{\text{av}} = 0$ of~\eqref{eq:error_dynamics_av_prob}, consider the following Lyapunov function candidate:
\begin{align}
V(\tilde{\theta}_{\text{av}}[k]) &= \tilde{\theta}_{\text{av}}^2[k]
\end{align}
By evaluating it along the trajectories of~\eqref{eq:error_dynamics_av_prob}, we obtain:
\begin{align}
V(\tilde{\theta}_{\rm {av}}[k+1])  
&= \alpha V(\tilde{\theta}_{\rm {av}}[k]),
\end{align}
with $\alpha = (1 - (1-p)hHK)^2$.  Under Assumption~\ref{assump:av_stability_prob}, it follows that $0 < \alpha < 1$ and, consequently, 
\begin{align}
    V(\tilde{\theta}_{\rm av}[k]) \leq \alpha^{k} V(\tilde{\theta}_{\rm av}[0]).
\end{align}
Therefore, we obtain
\begin{align}
    |\tilde{\theta}_{\rm av}[k]| \leq \alpha^{k/2} |\tilde{\theta}_{\rm av}[0]|, \label{eq:theta_av_convergence_prob}
\end{align}
which ensures the origin of the average system~\eqref{eq:error_dynamics_av_prob} is exponentially stable.

\subsubsection*{Almost surely and in probability convergence of the original stochastic system}

Since the vector field~\eqref{eq:F_sigma} is a continuous function of $\tilde{\theta}$ and $\xi = 1 - \sigma$, for any $\tilde{\theta} \in \mathbb{R}$ it is a bounded function of $\xi$, and it satisfies the locally Lipschitz condition in $\tilde{\theta} \in \mathbb{R}$ uniformly in $\xi$ satisfying Assumption~\ref{assump:probability}, it follows from~\cite[Lemma~7]{liu2015stochastic} that
\begin{align}
    \lim_{h \to 0} \sup_{0 \leq k \leq \lfloor N/h \rfloor} |\tilde{\theta}[k] - \tilde{\theta}_{\text{av}}[k]| = 0 \quad \rm{a.s.}
\end{align}
Then, from~\cite[Theorem~8]{liu2015stochastic}, we ensure that, for any $\delta > 0$ and $N \in \mathbb{N}$:
\begin{align}
    &\lim_{h \to 0} \inf\{k \in \mathbb{N}: |\tilde{\theta}[k] - \tilde{\theta}_{\text{av}}[k]| > \delta\} = +\infty~\rm{a.s.} \label{eq:error_delta} \\
    &\lim_{h \to 0} \mathbb{P}\left\lbrace \sup_{0 \leq k \leq \lfloor N/h \rfloor} |\tilde{\theta}[k] - \tilde{\theta}_{\text{av}}[k]| > \delta \right\rbrace = 0. \label{eq:error_delta_prob}
\end{align}

From the triangle inequality, we obtain $|\tilde{\theta}[k]| \leq |\tilde{\theta}_{\text{av}}[k]| + |\tilde{\theta}[k] - \tilde{\theta}_{\text{av}}[k]|$. From~\eqref{eq:theta_av_convergence_prob}, \eqref{eq:error_delta},~\eqref{eq:error_delta_prob}, yields directly \eqref{eq:theta_convergence_prob_k}--\eqref{eq:theta_convergence_prob}.
These results imply that $|\tilde{\theta}[k]|$ exponentially converges, both almost surely and in probability, to below an arbitrarily small residual value $\delta$ over an arbitrarily long time interval, which tends to infinity as $h \to 0$~\cite{liu2015stochastic}. 

To prove the output convergence to the extremum, for any $h > 0$, consider a stopping time $\tau_h^\delta$ in~\eqref{eq:tau_h}. Then by~\eqref{eq:theta_convergence_prob_k} we have that $\lim_{h \to 0} \tau_h^\delta = +\infty$, a.s. and
\begin{align}
|\tilde{\theta}[k]| \le |\tilde{\theta}[0]|\alpha^{k/2} + \delta, \forall k < \tau_h^\delta.  \label{eq:ultimate_theta}
\end{align}

Since $y[k] = f^\ast + \frac{H}{2} (\theta_a[k] - \theta^\ast)^2$ as in~\eqref{eq:y_eq}, due to the dependence on the operation mode associated with $\sigma[k] = 1 - \xi[k]$, we proceed by evaluating~\eqref{eq:y_eq} for both cases. If the communication is successful and $k \in \mathcal{Z}_i$, with $\mathcal{Z}_i$ defined as in~\eqref{eq:Zi}, we have $\sigma[k] = 1$ and
\begin{align}
    |y[k] - f^*| &= \frac{|H|}{2} |\theta[k] -\theta^*|^2 = \frac{|H|}{2} |\tilde{\theta}[k] + S[k]|^2 \nonumber \\
    &\leq C_h \alpha^{k}|\tilde{\theta}[0]|^2 + \mathcal{O}(a^2) + \mathcal{O}(\delta^2), \; \forall k \in \mathcal{Z}_i,\label{eq:y_ineq_1_prob}
\end{align}
whenever $k < \tau_h^\delta$. To express this bound in terms of the initial output error, we use $|\tilde{\theta}[0]|^2 = |\theta[0] - \theta^* - S[0]|^2$ and the inequality $(c-d)^2 \leq 2c^2 + 2d^2$ to obtain:
\begin{align}
    |\tilde{\theta}[0]|^2 &\leq 2|\theta[0] - \theta^*|^2 + 2|S[0]|^2 \nonumber \\
    &\leq \frac{4}{|H|}|y[0] - f^*| + \mathcal{O}(a^2).
\end{align}
Substituting this into~\eqref{eq:y_ineq_1_prob} and defining $C_y = 4 C_h / |H|$, the exponentially decaying $\mathcal{O}(a^2)$ term is absorbed into the steady-state bound, yielding:
\begin{align}
    |y[k] - f^*| &\leq C_y \alpha^{k}|y[0] - f^*| + \mathcal{O}(a^2) + \mathcal{O}(\delta^2), \; \forall k \in \mathcal{Z}_i.
\end{align}

By following similar arguments as in the proof of~Theorem~\ref{thm:stab_dos}, we can obtain that this inequality also holds for $k \in \mathcal{W}_i$, with $\mathcal{W}_i$ in~\eqref{eq:Wi}, since the output remains physically frozen at the last successfully transmitted value. Thus, \eqref{eq:y_convergence_prob_k} holds for all $k < \tau_h^\delta$.

Similarly, by combining this result with the probability limit in~\eqref{eq:theta_convergence_prob}, we obtain~\eqref{eq:y_convergence_prob}, which implies that the output can exponentially approach the extremum $f^\ast$ with probability 1 over finite horizons if $a$ is chosen sufficiently small. This concludes the proof.
\end{proof}

\section{Dither-Synchronized DoS ATTACKS on Extremum Seeking Control}
\label{sec:correlation_analysis}

In both the deterministic and stochastic frameworks, the exponential convergence rates scale inversely with the attack intensity, directly parameterized by the maximum attack duration fraction $\nu \in [0,1)$ and the attack success probability $p \in [0,1)$, as shown in Theorem~\ref{thm:stab_dos} and Theorem~\ref{thm:stab_dos_prob} for deterministic and probabilistic DoS attacks, respectively. While a higher frequency of unsuccessful transmissions due to DoS attacks reduces the convergence rate, the closed-loop system remains stable as long as the attacker cannot achieve permanent network blockage.

However, in a classical extremum-seeking control setup without an adequate resilient strategy, the malicious agent can launch a disruptive DoS attack that damages operation and drives the closed-loop system away from the optimum operating point $(\theta^\ast,f^\ast)$. For instance, consider a DoS attack affecting the controller-to-actuator channel that forces the actuator input to zero when communication is blocked ($\theta_a[k] = 0$), while the sensor continues to transmit the actual map output to the controller ($y_a[k] = y[k]$). In this case, the closed-loop system is subject to $\theta_a[k] = \sigma[k] \theta[k]$ and $y_a[k] = y[k]$, where $\sigma[k]$ is the signal that indicates whether the malicious agent blocks communication $\sigma[k] = 0$ or not $\sigma[k] = 1$. 
In this situation, the closed-loop system becomes
\begin{align}
    \tilde{\theta}[k+1] = \tilde{\theta}[k] + h G_{\sigma}(k,\tilde{\theta},h)
\end{align}
where
\begin{align} 
    G_{\sigma}(k,\tilde{\theta},h) &= - K M[k] f^* - \frac{K H}{2} \sigma[k] M[k] \tilde{\theta}^2[k] \nonumber \\ 
    &- \frac{K H}{2} \sigma[k] M[k] \left( 2\tilde{\theta}[k] S[k] + S[k]^2 \right) \nonumber \\ 
    &- \frac{K H}{2} M[k] (1 - \sigma[k])(\theta^*)^2, 
\end{align}
which depends on a small parameter $h$ and the vector field $G_{\sigma}(k,\tilde{\theta},h)$ is $T$-periodic in $k$, with $T$ given as in~\eqref{eq:T}.
Applying the discrete-time averaging technique yields the average system
\begin{align} 
\tilde{\theta}_{\text{av}}[k+1] &= \tilde{\theta}_{\text{av}}[k] + h G_{\sigma}^{\text{av}}(k,\tilde{\theta}_{\text{av}},h), \\ 
G_{\sigma}^{\text{av}}(k,\tilde{\theta}_{\text{av}},h) &= \lim_{h \to 0}\lim_{T \to \infty} \frac{1}{T} \sum_{k=s+1}^{s+T} G_{\sigma}(k,\tilde{\theta}_{\text{av}},h), \end{align}
where
\begin{align}
    \tilde{\theta}_{\text{av}}[k+1] &= \tilde{\theta}_{\text{av}}[k] 
    - h K H \frac{c_1}{a} \left( \tilde{\theta}_{\text{av}}^2[k] - (\theta^*)^2 \right) \nonumber \\
    & - h K H\left( 2 c_2 \tilde{\theta}_{\text{av}}[k] + a c_3 \right),
\end{align}
with 
\begin{align}
    c_1 = \lim_{h \to 0}\lim_{T \to \infty} \frac{1}{T} \sum_{k=s+1}^{s+T} \sigma[k] \sin(\omega h k) \label{eq:c1}\\
    c_2 = \lim_{h \to 0}\lim_{T \to \infty} \frac{1}{T} \sum_{k=s+1}^{s+T} \sigma[k] \sin^2(\omega h k)\\
    c_3 = \lim_{h \to 0}\lim_{T \to \infty} \frac{1}{T} \sum_{k=s+1}^{s+T} \sigma[k] \sin^3(\omega h k)
\end{align}
are the cross-correlation coefficients. 
Unlike the resilient strategies proposed in Sections~\ref{sec:main-results-1} and~\ref{sec:main-results-2}, the signal $\sigma[k]$ remains coupled with the demodulation signal $M[k]$, since it cannot be isolated and considered as a switch state. As ESC requires the dither amplitude $a$ to be sufficiently small, any attack sequence that correlates with the dither, \textit{i.e.}, $c_1 \neq 0$, introduces a bias of order $\mathcal{O}(1/a)$ that dominates the local dynamics. Thus, the steady-state equilibrium of the error $\tilde{\theta}_{\text{eq}}$ is
$$\frac{c_1}{a} \left( \tilde{\theta}_{\text{eq}}^2 - (\theta^*)^2 \right) \approx 0 \implies \tilde{\theta}_{\text{eq}} \approx \pm \theta^*$$
Consequently, the optimal parameter $\theta^*$ is not an equilibrium point, since $\tilde{\theta} = 0$ does not satisfy the averaged equilibrium condition when $c_1 \neq 0$. Instead, the parameter $\theta[k] = \tilde{\theta}[k] + \theta^*$ converges toward two attractors:
$\theta_{\text{eq}} = 0$, when the attack signal is correlated with the positive cycle of the dither and the cross-correlation is $c_1 < 0$; and $\theta_{\text{eq}} = 2\theta^*$, when the attack signal is correlated with the negative dither cycle and $c_1 > 0$. Thus, this strategy cannot ensure the system converges to its optimum point in the presence of DoS attacks, unlike the resilient approaches established in the previous sections.

\section{Numerical Examples}
\label{sec:results}

This section presents numerical examples to illustrate the operation of the proposed resilient extremum seeking control strategies against deterministic and probabilistic DoS attacks. For this purpose, consider a static nonlinear map parameterized by
\begin{align}
    \theta^\ast = 5, \quad f^\ast = 0, \quad H = 2.
\end{align}
The control parameters are selected as $a = 0.1$, $\omega = 10$ rad/s, $h = 0.1$ s, and the controller gain $K = 0.05$, satisfying the stability condition $|1 - h K H| < 1$ from Assumption~\ref{assump:av_stability}.

\subsection{Operation Under Deterministic DoS Attacks}

The simulations of the closed-loop system under deterministic DoS attacks with attack fractions $\bar{\nu} \in \{0.2, 0.7, 0.9\}$ are shown in Fig.~\ref{fig:deterministic_simulation}. In particular, Fig.~\ref{fig:deterministic_simulation}(a) presents the system output $y[k]$, while the Fig.~\ref{fig:deterministic_simulation}(b) presents the input parameter $\theta[k]$. The optimal reference values $f^*$ and $\theta^*$, respectively, are indicated by dashed lines.

\begin{figure}[!ht]
    \centering
    \begin{subfigure}[b]{\columnwidth}
         \centering
         \includegraphics[width=\textwidth]{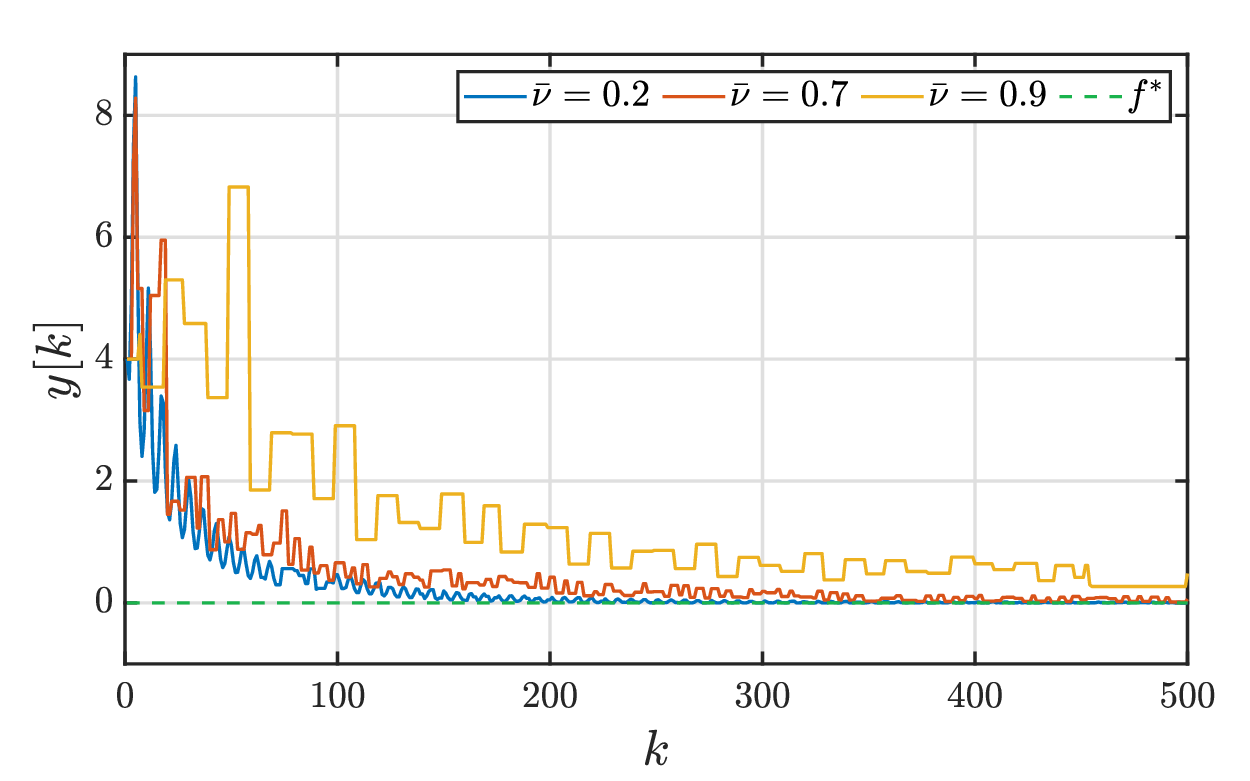}
         \caption{Output of the nonlinear map --- $y[k]$.}
     \end{subfigure}
     \begin{subfigure}[b]{\columnwidth}
         \centering
         \includegraphics[width=\textwidth]{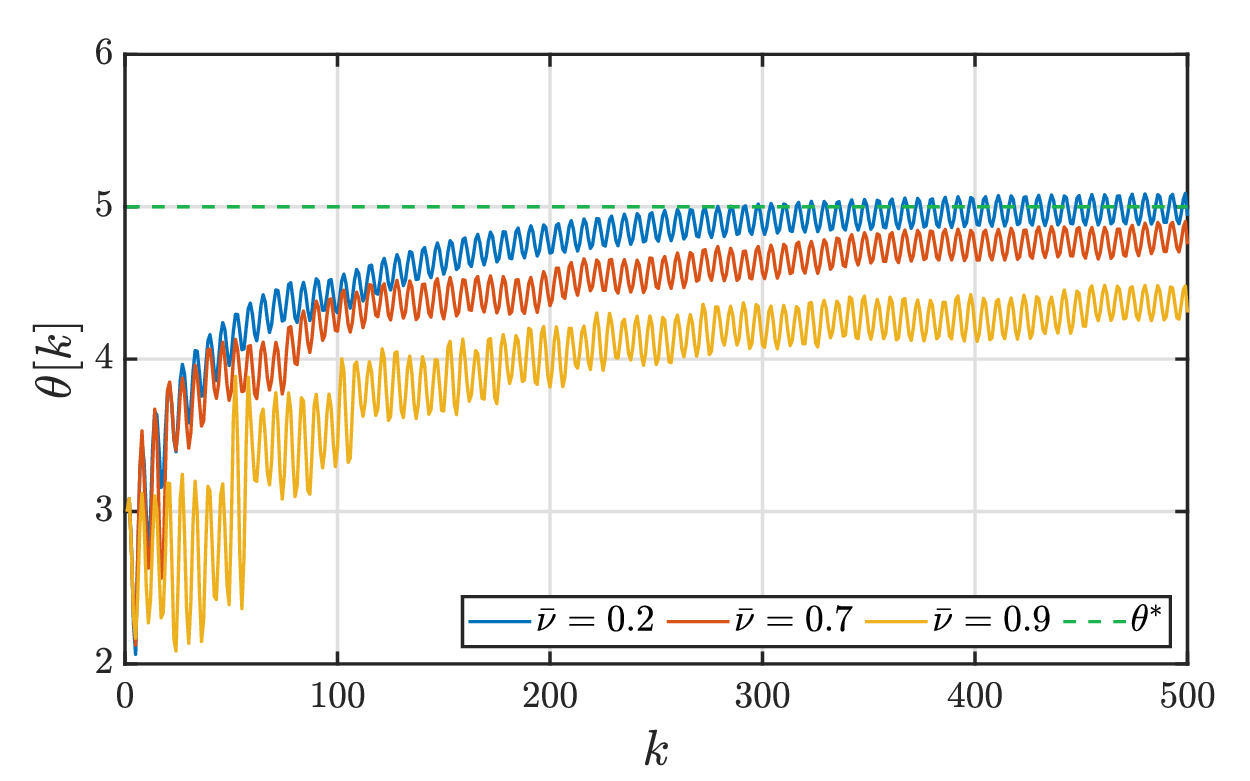}
         \caption{Input transmitted to the nonlinear map --- $\theta[k]$.}
     \end{subfigure}
    \caption{Simulation of the closed-loop system under the proposed resilient strategy against deterministic DoS attacks with different attack bounds $\bar{\nu}$.}
    \label{fig:deterministic_simulation}
\end{figure}

For $\bar{\nu} = 0.2$ (blue lines), the closed-loop system converges faster than in the other cases toward the extremum. With larger attack fractions of $\bar{\nu} = 0.7$ (red lines) and $\bar{\nu} = 0.9$ (yellow lines), the convergence rates become slower. These results illustrate that the proposed resilient strategy guarantees practical exponential stability for any $\nu < 1$, as established in Theorem~\ref{thm:stab_dos}, although the convergence rate is affected as the attack fraction increases. However, the residual error remains bounded within the predicted orders $\mathcal{O}(a + h)$ for $\theta[k]$ and $\mathcal{O}(a^2 + h^2)$ for $y[k]$, showing the effectiveness of the resilient scheme even under severe attack conditions.

\subsection{Operation Under Probabilistic DoS Attacks}

The simulations of the closed-loop system under probabilistic DoS attacks with attack probabilities $p \in \{0.2, 0.7, 0.9\}$ are shown in Fig.~\ref{fig:probabilistic_simulation}, where Fig.~\ref{fig:probabilistic_simulation}(a) presents the system output $y[k]$ and Fig.~\ref{fig:probabilistic_simulation}(b) the input parameter $\theta[k]$. Note that Assumption~\ref{assump:av_stability_prob} is satisfied for all considered attack probabilities.

\begin{figure}[!ht]
    \centering
    \begin{subfigure}[b]{\columnwidth}
         \centering
         \includegraphics[width=\textwidth]{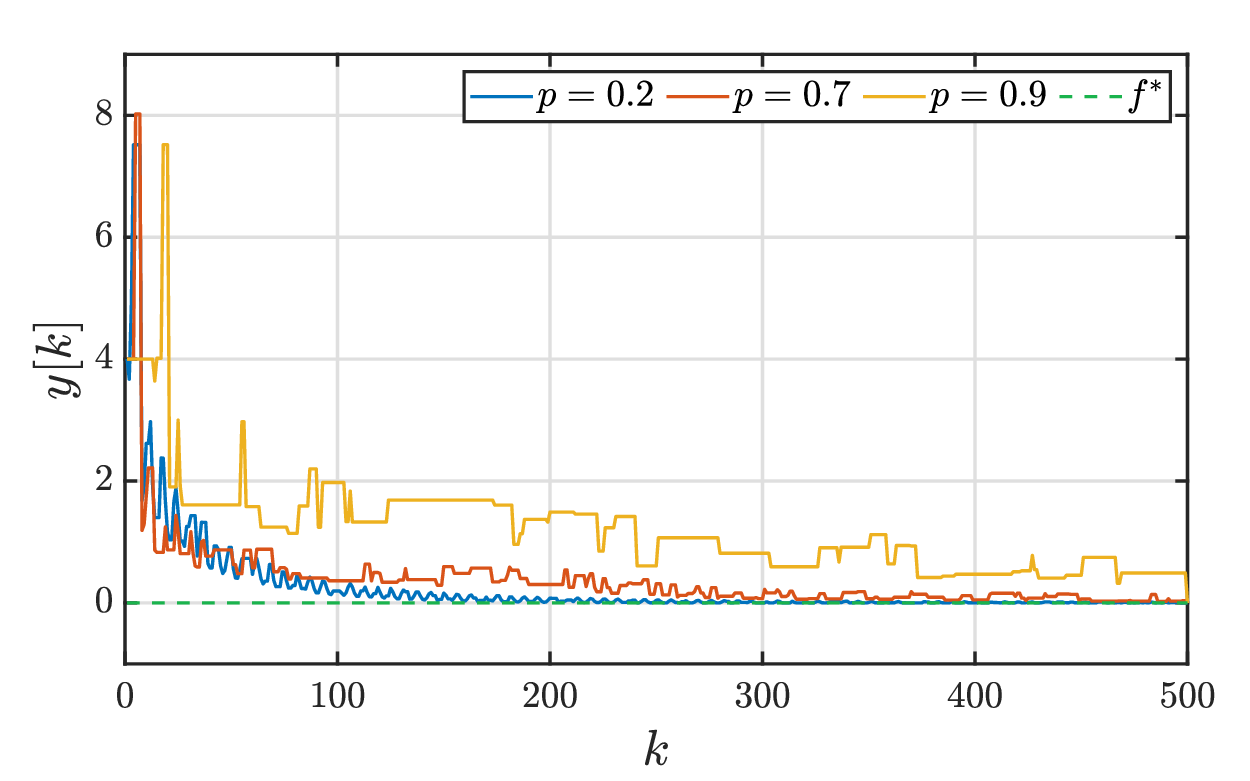}
         \caption{Output of the nonlinear map --- $y[k]$.}
     \end{subfigure}
     \begin{subfigure}[b]{\columnwidth}
         \centering
         \includegraphics[width=\textwidth]{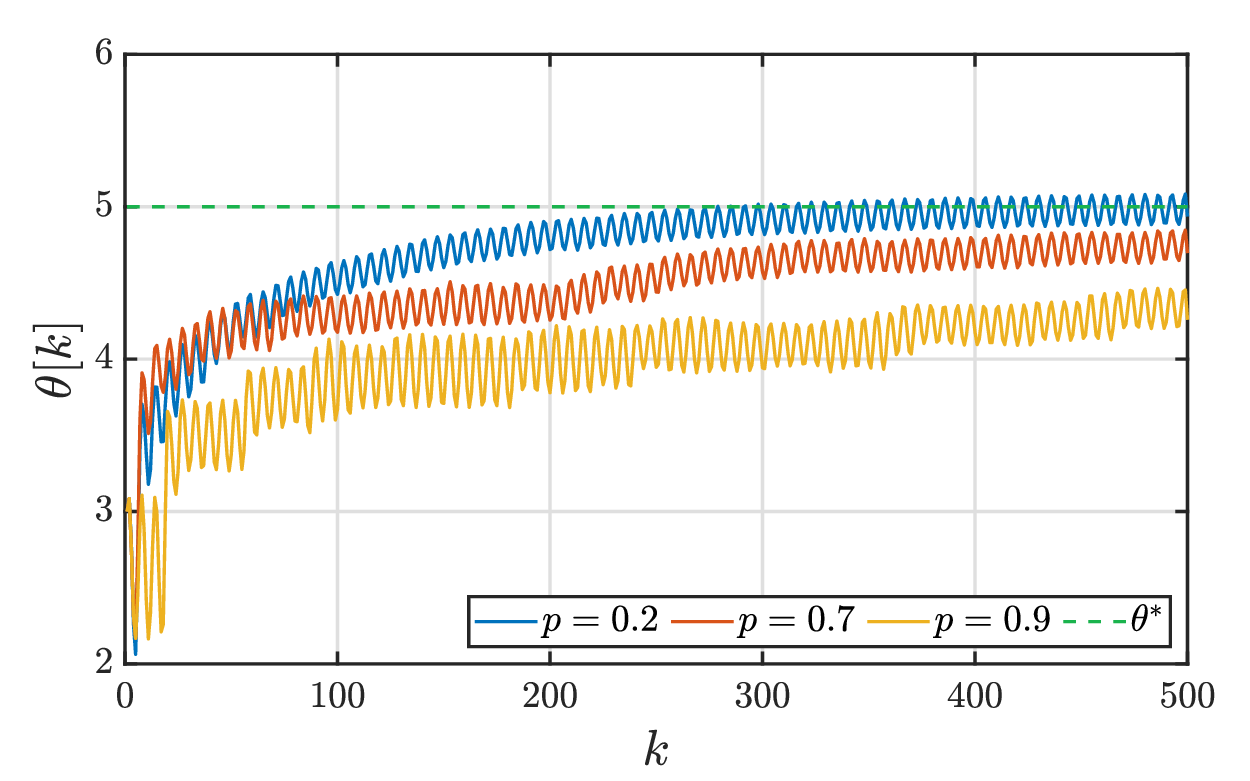}
         \caption{Input transmitted to the nonlinear map --- $\theta[k]$.}
     \end{subfigure}
    \caption{Simulation of the closed-loop system under the proposed resilient strategy against probabilistic DoS attacks with different attack probabilities $p$.}
    \label{fig:probabilistic_simulation}
\end{figure}

As observed in the deterministic case, the convergence rate reduces as the attack becomes more severe. These results are consistent with simulations using the deterministic DoS model, shown in~Fig.~\ref{fig:deterministic_simulation}, reinforcing the conclusion that the proposed resilient scheme ensures practical exponential stability for any attack probability $p < 1$, as established in Theorem~\ref{thm:stab_dos_prob}.

\subsection{Effect of dither-synchronized DoS attacks}

In this section, we compare the hold-input and zero-input resilient strategies against deterministic DoS attacks. For this purpose, we consider $\bar{\nu} = 0.5$ and three attack sequences:
\begin{itemize}
    \item[$(i)$] Arbitrary attack satisfying Assumption~\ref{assump:duration};
    \item[$(ii)$] Attack synchronized with the positive cycles of the dither signal:
    \begin{align}\label{eq:attack_postive}
        \xi[k] = \begin{cases}
            1, & \mathrm{if} \; \sin{(\omega h k) > 0}, \\
            0  & \mathrm{otherwise}.
        \end{cases}
    \end{align}
    \item[$(iii)$] Attack synchronized with the negative cycles of the dither signal:
    \begin{align}\label{eq:attack_negative}
        \xi[k] = \begin{cases}
            1, & \mathrm{if} \; \sin{(\omega h k) < 0}, \\
            0  & \mathrm{otherwise}.
        \end{cases}
    \end{align}
\end{itemize}
As the considered angular frequency of the dither signals is $\omega = 10~\mathrm{rad/s}$, the three attack sequences satisfy the fraction of time under attack of $\nu < \bar{\nu} = 0.5$. The input of the nonlinear map under the three attack sequences is shown in Fig.~\ref{fig:comparison_simulation}, using the hold- and zero-input strategies. In all three cases, the system converges to the optimum point with the hold-input strategy, but not with the zero-input strategy. 
\begin{figure}[!ht]
    \centering
    \begin{subfigure}[b]{\columnwidth}
         \centering
         \includegraphics[width=\textwidth]{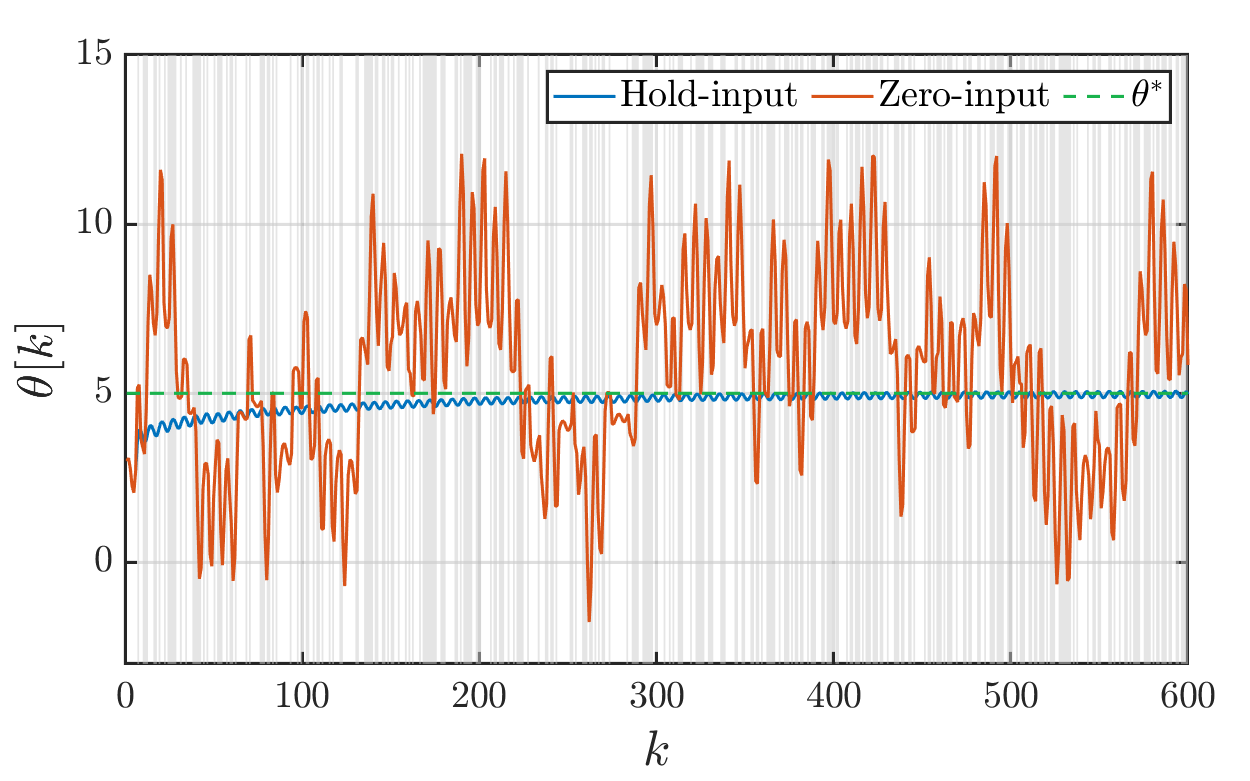}
         \caption{Arbitrary DoS attack satisfying Assumption~\ref{assump:duration} with $\bar{\nu} = 0.5$.}
     \end{subfigure}
     \begin{subfigure}[b]{\columnwidth}
         \centering
         \includegraphics[width=\textwidth]{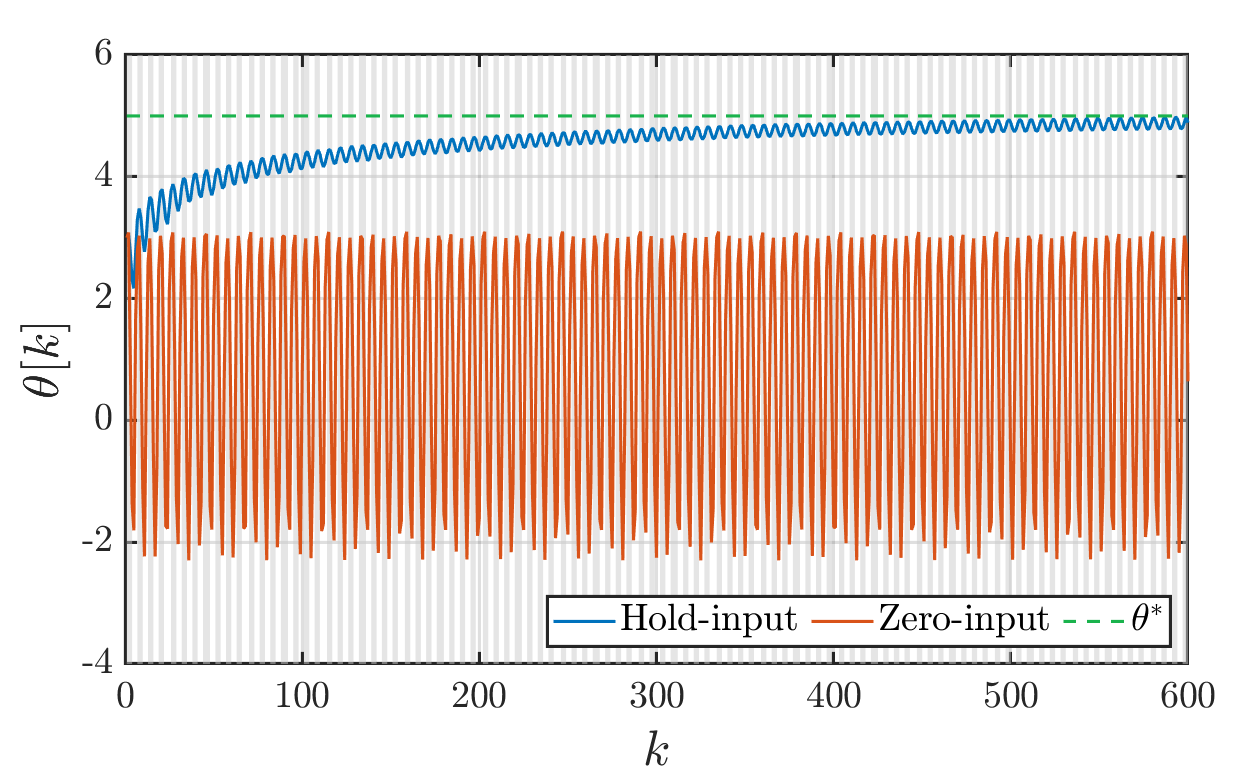}
         \caption{DoS attack synchronized with the positive cycles of the dither signal.}
     \end{subfigure}
     \begin{subfigure}[b]{\columnwidth}
         \centering
         \includegraphics[width=\textwidth]{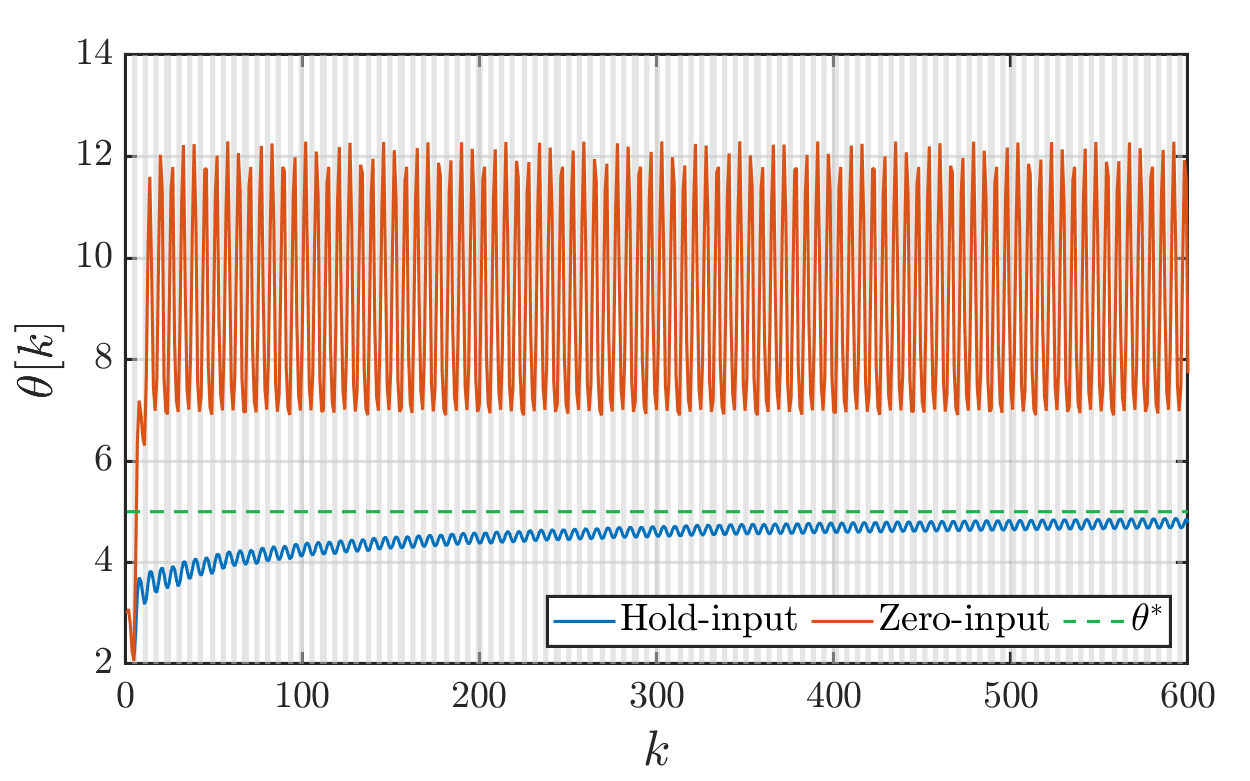}
         \caption{DoS attack synchronized with the negative cycles of the dither signal.}
     \end{subfigure}
    \caption{Comparison between the use of the hold- and zero-input strategies in $\theta_a[k]$ with nominal measurement $y_a[k] = y[k]$ under different DoS attack sequences. The gray stripes indicate the samples under DoS attacks.}
    \label{fig:comparison_simulation}
\end{figure}

In particular, the effect of the correlation between the attack signal $\sigma[k]$ and the dither signal on the closed-loop behavior under the zero-input strategy is illustrated in the simulations shown in Fig.~\ref{fig:comparison_simulation}(b) and ~\ref{fig:comparison_simulation}(c). As discussed in Section~\ref{sec:correlation_analysis}, when the attack sequence is correlated with the dither, the cross-correlation coefficient $c_1$ in~\eqref{eq:c1} becomes nonzero. This shifts the equilibrium point away from the true optimum $\theta^*$, as evidenced by the steady-state value of $\theta[k]$ converging to $\theta_{\text{eq}} = 0$ for $c_1 < 0$, in Fig.~\ref{fig:comparison_simulation}(b), or $\theta_{\text{eq}} = 2\theta^*$ for $c_1 > 0$, in Fig.~\ref{fig:comparison_simulation}(c). In contrast, the hold-input strategy effectively decouples the attack signal from the dither, thereby eliminating the bias and ensuring convergence to the true extremum $\theta^*$. These results confirm that the hold-input strategy effectively prevents the harmful effects of correlated attacks, which can compromise closed-loop convergence.

\section{Conclusion}
\label{sec:conclusion}

This paper has investigated resilient extremum seeking control for
networked cyber-physical systems subject to denial-of-service attacks.
The proposed hold-based communication strategy reveals a particularly
useful structure in the averaged closed-loop dynamics: successful
transmissions generate an exponentially contracting mode, whereas
communication losses generate a neutral mode that preserves the most
recently available information. This structure allowed resilience to be
established under both deterministic and probabilistic DoS models. In
the deterministic setting, practical exponential convergence was shown
under an average bound on the attack duration, while in the
probabilistic setting, almost-sure and in-probability convergence
properties were obtained through stochastic averaging. In both cases,
the analysis explicitly shows how increasing attack intensity slows the
optimization process without destroying convergence, provided that
communication is not permanently blocked.

Perhaps the central insight emerging from this work is that DoS attacks
against extremum-seeking loops cannot be understood solely in terms of
how much information is lost. In perturbation-based ESC, an attacker can
exploit the correlation between packet losses and the probing signal to
bias the very mechanism by which optimization information is extracted.
The analysis developed here shows that holding the last successfully
transmitted signal structurally removes this vulnerability, turning DoS
intervals into neutral, rather than destabilizing, modes of the averaged
optimization dynamics. Hence, resilience in networked ESC depends not
only on \emph{how often} communication is denied, but also on \emph{how
the losses interact with the excitation mechanism that drives the
search}. This distinction---between loss of information and corruption
of the search mechanism itself---is the main conceptual message of this
work.

Several directions emerge from these results. Extending the analysis
from static maps to nonlinear dynamic plants would clarify how DoS
attacks interact simultaneously with the plant dynamics and the
extremum-seeking time scales. Multivariable ESC introduces additional
possibilities for attacks correlated with multiple probing frequencies
and therefore raises new questions concerning attack-excitation
coupling. More general stochastic and adversarial attack models,
including asynchronous losses in the sensor-to-controller and
controller-to-actuator channels, are also of interest. Another promising
direction is the development of resilient
networked and game-theoretic architectures building on event-triggered
and update-rate-constrained Nash equilibrium seeking
\cite{rodrigues2026distributed,TAC2026bounded}, as well as extensions
to systems with delays and distributed dynamics
\cite{oliveira2026extremum}. 

\bibliographystyle{IEEEtran}
\bibliography{references}

\begin{IEEEbiography}[{\includegraphics[width=1in,height=1.25in,clip,keepaspectratio]{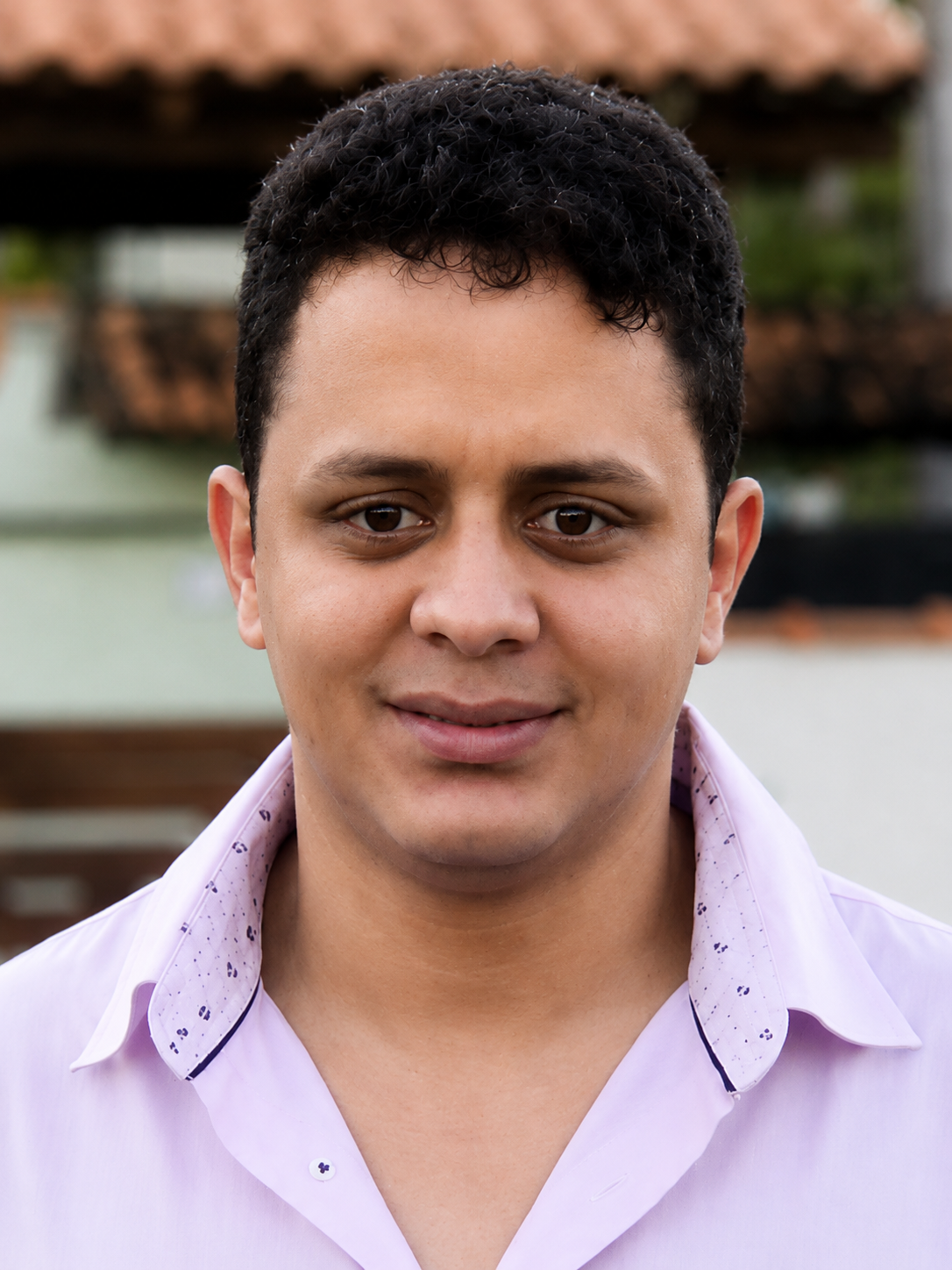}}]{Filipe da Silva Bastos Teixeira} received the B.Sc. degree in mathematics from Fluminense Federal University (UFF), Brazil, in 2018, and the B.Sc. degree in electrical engineering from Dom Bosco University, Brazil, in 2023. He is currently pursuing the M.Sc. degree in electronic engineering with the State University of Rio de Janeiro (UERJ), Rio de Janeiro, Brazil. His research interests include extremum seeking and networked control systems.
\end{IEEEbiography} 
\begin{IEEEbiography}[{\includegraphics[width=1in,height=1.25in,clip,keepaspectratio]{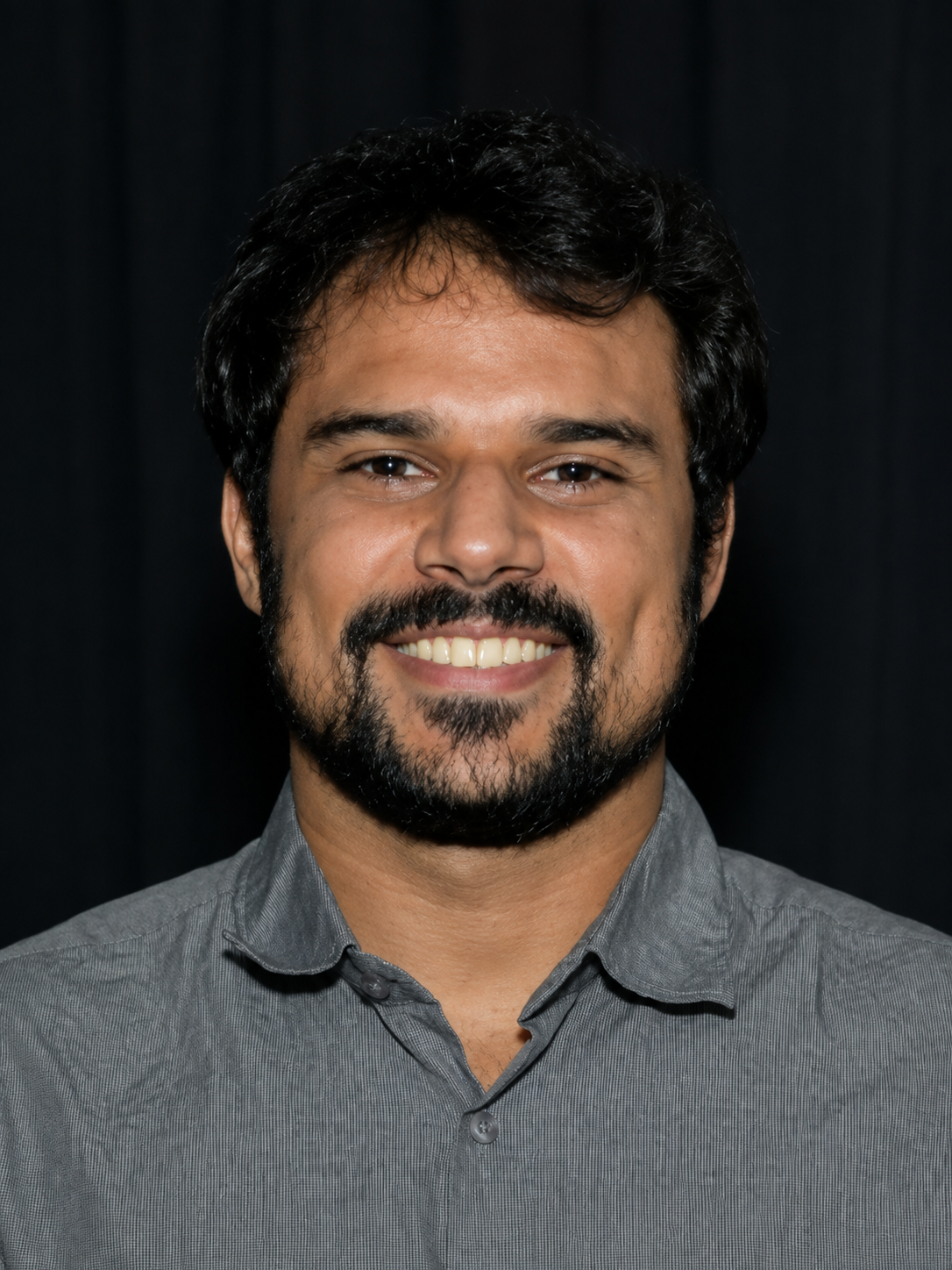}}]{ Pedro~Henrique~Silva~Coutinho} is currently a Professor at the Rio de Janeiro State University (UERJ), Rio de Janeiro, Brazil. Prof. Coutinho has been serving as Associate Editor for Fuzzy Sets and Systems since 2024 and for the Journal of Control, Automation and Electrical Systems and the International Journal of Robust and Nonlinear Control since 2026. He received the IFAC Young Author Award at the 4th Conference on Embedded Systems, Computational Intelligence, and Telematics in Control (CESCIT 2021). His research interests include robust and nonlinear control, fuzzy systems, data-driven control, cyber-physical systems, and industrial process monitoring.
\end{IEEEbiography} 
\begin{IEEEbiography}[{\includegraphics[width=1in,height=1.25in,clip,keepaspectratio]{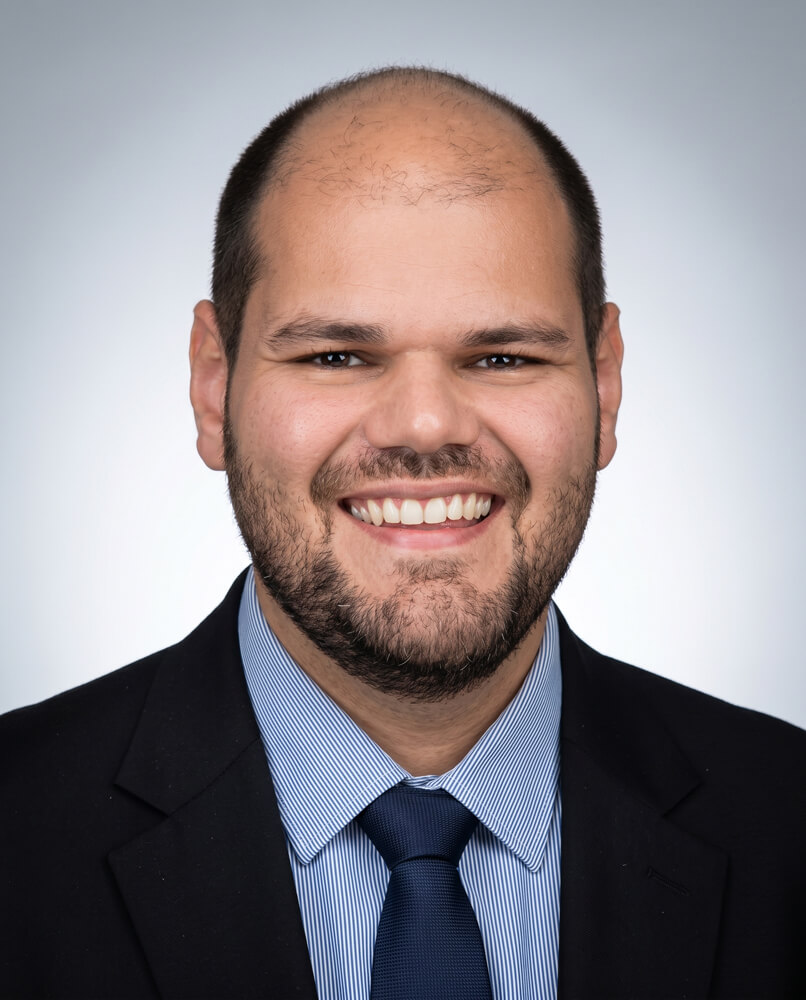}}]{Tiago Roux Oliveira} (Senior Member, IEEE) 
is currently Professor at State University of Rio de Janeiro and author of over 300 publications.
He serves as Associate Editor for several control journals (including IEEE Transactions on Automatic Control, IEEE Control Systems Letters, and IEEE Open Journal of Control Systems), chaired the IFAC TC on Adaptive and Learning Systems (2020–2026), and was elected President of the Brazilian Society of Automatics (2023–2025). His awards include 
the 2021 IEEE TCST Outstanding Paper Award. 
He is the coauthor of the book \textit{Extremum Seeking through Delays and PDEs} (SIAM, 2022), which presents a unified treatment of extremum seeking for systems with delays, partial differential equations, and distributed optimization.
\end{IEEEbiography}
\begin{IEEEbiography}[{\includegraphics[width=1in,height=1.25in,clip,keepaspectratio]{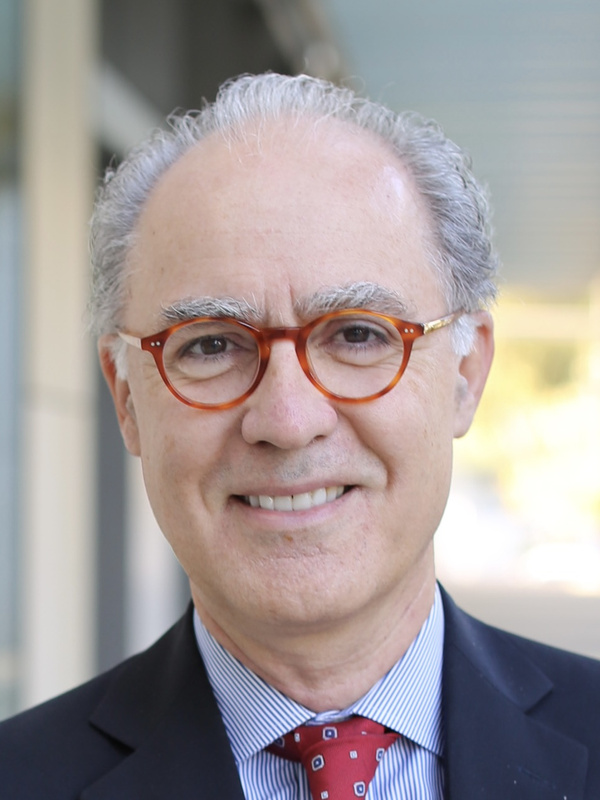}}]{Miroslav Krstic} is Fellow of SIAM, IEEE, IFAC, ASME, AAAS, IET, AIAA (AF), and Serbian Academy of Sciences and Arts.
His awards include the SIAM Reid Prize, Bellman, Oldenburger, Ragazzini, Chestnut, Paynter, Nyquist Lecture, Bode Lecture, IFAC Nonlinear Control,
IFAC Ruth Curtain DPS, IFAC Adaptive and Learning Systems, Axelby, and Schuck (’96 and ’19). He has held chief or senior editorial positions for IEEE Transactions on Automatic Control, Systems \& Control Letters, and Automatica. Professor Krstic has co-authored 19 books on adaptive, nonlinear, and
stochastic control, extremum seeking, control of PDE systems including turbulent flows, and control of delay systems.
\end{IEEEbiography}

\end{document}